\documentclass[12pt,a4paper]{article}
\usepackage[utf8]{inputenc}
\usepackage{amsmath}
\usepackage{amsfonts}
\usepackage{amssymb}
\usepackage{amsthm}
\usepackage{graphicx}
\usepackage{verbatim}
\usepackage{hyperref}
\usepackage{caption}
\usepackage{subcaption}
\usepackage{color}

\newtheoremstyle{firststyle}  
  {5mm}       
  {3mm}       
  {\itshape}   
  {}        
  {\bfseries}  
  {}   
  {3mm}       
  {}           

\theoremstyle{firststyle}
\newtheorem{theorem}{Theorem}[section]
\newtheorem{lemma}[theorem]{Lemma}

\newtheorem{corollary}[theorem]{Corollary}

\newtheoremstyle{secondstyle}  
  {3mm}       
  {3mm}       
  {\normalfont}   
  {}        
  {\bfseries}  
  {\quad}   
  {1mm}       
  {}           

\theoremstyle{secondstyle}

\newtheorem{remark}{Remark}

\newtheorem{example}{Example}
\newtheorem{algorithm}{Algorithm}

\usepackage{booktabs} 
\title{Reduced basis methods\textendash an application to variational discretization of parametrized elliptic optimal control problems} 
\author{Ahmad Ahmad Ali\footnote{Schwerpunkt Optimierung und Approximation, 
		Universit\"at Hamburg, Bundesstra{\ss}e 55, 20146 Hamburg, Germany.} \& Michael Hinze\footnote{Schwerpunkt Optimierung und Approximation, 
		Universit\"at Hamburg, Bundesstra{\ss}e 55, 20146 Hamburg, Germany.}}

\date{}
\begin{document}
\maketitle

{\small {\bf Abstract:} We consider a class of parameter-dependent optimal control problems of elliptic PDEs with constraints of general type on the control variable. Applying the concept of variational discretization, \cite{hinze2005variational}, together with  techniques from the Reduced basis method, we construct a reduced basis surrogate model for the control problem. We establish estimators for the greedy sampling procedure which only involve the residuals of the state and the adjoint equation, but not of the gradient equation of the optimality system. The estimators are sharp up to a constant, i.e. they are equivalent to the approximation erros in control, state, and adjoint state. Numerical experiments show the performance of our approach.}

\section{Introduction}
The research in this work is motivated by the reduced basis approaches of \cite{bader2016certified} applied to approximate the solution manifold of the parameter dependent control constrained optimal control problem \eqref{P}. The approach taken there uses a fully discrete treatment of the optimal control problem \eqref{P}, so that the constructed a posteriori error estimators involve the residuals of the state, of the adjoint and of the gradient equation of the corresponding optimality conditions. Since the gradient equation in the control constrained case is nonsmooth one expects large contributions of the control residual in the estimation process. Our approach uses variational discretization \cite{hinze2005variational} of \eqref{P} which avoids explicit discretization of the control variable, see problem \eqref{PN}. This approach then allows us to construct reliable and effective a posteriori error bounds only involving the residuals of the state and the adjoint state, respectively, see Theorem \ref{thm:uyp}. Moreover, in Corollary \ref{thm:u} we propose an estimator for the relative error in the controls which only involves the residuals of the state and the adjoint state. We test our approach at the numerical examples presented in  \cite{bader2016certified}. It is one important result of our work that the reduced basis spaces constructed with our approach for a given error level have much smaller dimensions than the respective spaces constructed with the approaches of \cite{bader2016certified}. In the present work we focus on the approximation quality of the reduced spaces constructed with our approach from the dimensionality point of view. We do not discuss questions related to offline-online decomposition in our approach.

We note that our numerical analysis related to the error equivalence of Theorem \ref{thm:uyp} is motivated by techniques frequently used in the convergence analysis of adaptive finite element methods for optimal control problems, see e.g. \cite{gong2017adaptive}. For excellent introductions to the reduced basis method for approximations of parameter dependent elliptic PDEs we refer the reader to \cite{haasdonk2017reduced,rozza2007reduced}. For a discussion of reduced basis approaches to approximate parameter dependent optimal control problems we refer the reader to \cite{bader2016certified}, where also further literature can be found, and also detailed discussions related to offline-online decomposition in the numerical implementation are provided.

\section{General setting}
Let $\mathcal{P}\subset \mathbb{R}^p$, $p \in \mathbb{N}$, be a compact set of parameters, and  for a given parameter $\mu \in \mathcal{P}$ we consider the  variational discrete (\cite{hinze2005variational}) control problem 
\begin{equation}\label{P}
(\mathbb{P}) \quad \begin{array}{l}
\min_{(u,y) \in U_{ad} \times Y} J(u,y):=\frac{1}{2} \| y-z \|_{L^2(\Omega_0)}^2  + \frac{\alpha}{2} \| u \|_U ^2 \\
\text{subject to}
\end{array}
\end{equation}
\begin{align}
a(y,v;\mu)= b(u,v;\mu) + f(v;\mu)   \quad \forall \, v \in Y. \label{wpde}
\end{align}
Here (\ref{wpde}) represents a finite element discrete elliptic PDE in a bounded domain $\Omega \subset \mathbb{R}^d$ for $d \in \{1, 2, 3\}$   with boundary $\partial \Omega$.  $Y$ denotes the space of piecewise linear and continuous finite elements. We assume the approximation process is conforming. The space $Y$ is equipped by the inner product $(\cdot,\cdot)_Y$ and the norm $\|\cdot\|_Y:=\sqrt{(\cdot,\cdot)_Y}$, in addition, there exist constants $\rho_1, \rho_2 >0$ such that there holds 
\begin{align}
\rho_1 \|y\|_{H^1(\Omega)} \leq \|y\|_{Y} \leq \rho_2 \|y\|_{H^1(\Omega)}  \quad \forall \, y \in Y, \label{norm_eqv}
\end{align}
with $\|\cdot\|_{H^1(\Omega)}$ being the norm of the classical Sobolev space $H^1(\Omega)$.

The controls are from a real Hilbert space $U$ equipped by the inner product $(\cdot,\cdot)_U$ and the norm $\|\cdot\|_U:=\sqrt{(\cdot,\cdot)}_U$, and the set of admissible controls $U_{ad} \subseteq U$ is assumed to be non-empty, closed and convex.

We denote by $\Omega_0 \subseteq \Omega$ an open subset, and $L^2(\Omega_0)$  the classical Lebesgue space endowed with the standard inner product $(\cdot,\cdot)_{L^2(\Omega_0)}$ and the  norm $\|\cdot\|_{L^2(\Omega_0)}:=\sqrt{(\cdot,\cdot)_{L^2(\Omega_0)}}$.The desired state $z \in L^2(\Omega_0)$ and the parameter  $\alpha >0$  are given data.

The parameter dependent bilinear form $a(\cdot,\cdot ;\mu): Y \times Y \to \mathbb{R}$  is continuous 
\begin{align*}
\gamma(\mu):=\sup_{y, v \in Y\backslash\{0\}} \frac{|a(y,v;\mu)|}{ \|y\|_{Y}  \|v\|_{Y}} \leq \gamma_0 < \infty \qquad \forall\, \mu \in \mathcal{P},
\end{align*}
and coercive 
\begin{align*}
 \beta(\mu):=\inf_{y \in Y\backslash\{0\}} \frac{a(y,y;\mu)}{ \|y\|^2_{Y}} \geq \beta_0 > 0 \qquad \forall\, \mu \in \mathcal{P},
\end{align*}
where $\gamma_0$ and $\beta_0$ are real numbers independent of $\mu$. The parameter dependent bilinear form $b(\cdot,\cdot;\mu): U \times Y \to \mathbb{R}$ is  continuous
\begin{align*}
\kappa(\mu):=\sup_{(u, v) \in U\times Y \backslash\{(0,0)\}} \frac{|b(u,v;\mu)|}{ \|u\|_{U}  \|v\|_{Y}} \leq \kappa_0 < \infty \qquad \forall\, \mu \in \mathcal{P},
\end{align*}
where $\kappa_0$ is a real number independent of $\mu$. Finally, $f(\cdot;\mu) \in Y^\ast$ is a parameter dependent linear form, where $Y^\ast$ denotes the topological dual of $Y$ with norm $\|\cdot\|_{Y^\ast}$ defined by
\begin{align*}
\|l(\cdot;\mu) \|_{Y^{\ast}}:= \sup_{\|v \|_Y=1} l(v;\mu),
\end{align*}
for a give functional $l(\cdot;\mu) \in Y^\ast$ depending on the parameter $\mu$. We assume that there exists a constant $\sigma_0$ independent of $\mu$ such that
\begin{align*}
\sup_{v \in Y \backslash\{0\}} \frac{|f(v;\mu)|}{ \|v\|_{Y}} \leq \sigma_0 < \infty \qquad \forall\, \mu \in \mathcal{P}.
\end{align*}

We find it convenient to  introduce here for the upcoming analysis the Riesz isomorphism  $R: Y^\ast \to Y$ which is defined for a given $f \in Y^\ast$ by the unique element $R f \in Y$ such that
\begin{align*}
f(v)=(R f, v)_Y  \quad \forall \, v \in Y.
\end{align*}

Under the previous assumptions one can verify that the problem $(\mathbb{P})$ admits a unique solution for every $\mu \in \mathcal{P}$. The corresponding first order necessary conditions, which are also sufficient in this case, are stated in the next result. For the proof see for instance \cite[Chapter~3]{hinze2008optimization}.
\begin{theorem}
\label{thm:oc}
Let $u \in U_{ad}$ be the solution of $(\mathbb{P})$ for a given $\mu \in \mathcal{P}$. Then there exist a state $y \in Y$ and an adjoint state $p \in Y$ such that there holds
\begin{align}
& a(y,v;\mu)=b (u, v;\mu) + f(v;\mu)  \quad \forall \, v \in Y, \label{os_y}\\
& a(v,p;\mu)=(y-z,v)_{L^2(\Omega_0)} \quad \forall \, v \in Y, \label{os_p}\\
& b(v-u,p;\mu)  + \alpha (u, v-u)_U \geq 0 \quad \forall \, v \in U_{ad}. \label{os_u}
\end{align}
\end{theorem}

The varying parameter $\mu$ in the state equation (\ref{wpde}) could represent physical or/and geometrical quantities, like diffusion or convection speed, or the width of the spacial domain $\Omega$. Considering the problem $(\mathbb{P})$ in the context of real-time or multi-query scenarios can be very costly when the dimension of the finite element space $Y$ is very high. In this work we adopt  the reduced basis method, see for instance \cite{haasdonk2017reduced}, to obtain a surrogate for $(\mathbb{P})$  that is relatively cheaper to solve and at the same time delivers acceptable approximation for the solution of  $(\mathbb{P})$ at a given $\mu$. To this end, we first define a reduced problem for  $(\mathbb{P})$, and establish a posterior error estimators that predict the expecting approximation error when using the reduced problem. Then, we apply a greedy procedure (see Algorithm~\ref{algo:greedy}) to improve the approximation quality of the reduced problem.

%

\section{The reduced problem and the greedy procedure}

Let $Y_N \subset Y$ be a finite dimensional subspace. We define the reduced  counterpart of the problem $(\mathbb{P})$ for a given $\mu \in  \mathcal{P}$ by
\begin{equation}\label{PN}
(\mathbb{P}_N) \quad \begin{array}{l}
\min_{(u,y_N) \in U_{ad} \times Y_N} J(u,y_N):=\frac{1}{2} \| y_N-z \|_{L^2(\Omega_0)}^2  + \frac{\alpha}{2} \| u \|_U ^2\\
\text{subject to}
\end{array}
\end{equation}
\begin{align}
a(y_N,v_N;\mu)= b(u,v_N;\mu) + f(v_N;\mu)   \quad \forall \, v_N \in Y_N.\label{Nwpde} 
\end{align}

We point out that in  $(\mathbb{P}_N)$ the controls are still sought in $U_{ad}$. In a similar way to $(\mathbb{P})$, one can show that $(\mathbb{P}_N)$ has a unique solution for a given $\mu$, and it satisfies the following  optimality conditions.
\begin{theorem}
Let $u_N \in U_{ad}$ be the solution of $(\mathbb{P}_N)$ for a given $\mu \in \mathcal{P}$. Then there exist a state $y_N \in Y_N$ and an adjoint state $p_N \in Y_N$ such that there holds 
\begin{align}
& a(y_N,v_N;\mu)=b (u_N, v_N;\mu) + f(v_N;\mu)  \quad \forall \, v_N \in Y_N, \label{os_yN}\\
& a(v_N,p_N;\mu)=(y_N-z,v_N)_{L^2(\Omega_0)} \quad \forall \, v_N \in Y_N, \label{os_pN}\\
& b(v-u_N,p_N;\mu)  + \alpha (u_N, v-u_N)_U \geq 0 \quad \forall \, v \in U_{ad}. \label{os_uN}
\end{align}
\end{theorem}

The space $Y_N$ shall be  constructed successively  using the following greedy procedure.
\begin{algorithm}[Greedy procedure]\label{algo:greedy}\mbox{ }
\\[2mm]
1. Choose $S_{\mbox{\scriptsize  train}} \subset \mathcal{P}$, $\mu^1 \in S_{\mbox{\scriptsize  train}}$ arbitrary, $\varepsilon_{\mbox{\scriptsize tol}} > 0$, and $N_{\mbox{\scriptsize max}} \in \mathbb{N}$.
\\[2mm]
2. Set $N=1$, $\Phi_1:=\{y(\mu^1),p(\mu^1)\}$, and $Y_1:=\mbox{span}(\Phi_1)$.
\\[2mm]
3. \textbf{while} $\max_{ \mu \in S_{\mbox{\scriptsize train}}}\Delta(Y_N,\mu) > \varepsilon_{\mbox{\scriptsize tol}}$ and $N \leq N_{\mbox{\scriptsize max}}$ \textbf{do}
\\[2mm]
4. \mbox{ \qquad } $\mu^{N+1}:= \mbox{arg max}_{ \mu \in S_{\mbox{\scriptsize train}}}\Delta(Y_N,\mu)$
\\[2mm]
5. \mbox{ \qquad } $\Phi_{N+1}:= \Phi_N \cup \{y(\mu^{N+1}),p(\mu^{N+1})\}$
\\[2mm]
6. \mbox{ \qquad } $Y_{N+1}:= \mbox{span}(\Phi_{N+1})$
\\[2mm]
7. \mbox{ \qquad } $N \leftarrow N+1$
\\[2mm]
8. \textbf{end while}
\end{algorithm}
Here $S_{\mbox{\scriptsize  train}} \subset \mathcal{P}$ is a finite subset, called a training set, which  assumed to be rich enough in parameters to well represent $\mathcal{P}$. $N_{\mbox{\scriptsize max}}$ is the maximum number of iterations, and  $\varepsilon_{\mbox{\scriptsize tol}}$ is a given error tolerance.
In the iteration of index $N$, the pair $\{y(\mu^N),p(\mu^N)\}$ is the optimal state and adjoint state, respectively, corresponding to the problem $(\mathbb{P})$ at  $\mu^N$, and $\Phi_N$ is the reduced basis which  assumed to be orthonormal. If it is not, one can apply an orthonormalization process like Gram--Schmidt. An orthonormal reduced basis guarantees algebraic stability when $N$ increases, see \cite{haasdonk2017reduced}. The quantity $\Delta(Y_N,\mu)$ is an estimator for the expected error in approximating the solution of $(\mathbb{P})$  by the one of $(\mathbb{P}_N)$ for a given $\mu$ when using the space $Y_N$. The maximum of $\Delta(Y_N,\mu)$ over $S_{\mbox{\scriptsize train}}$ is obtained by linear search. 

We note that at line 5 in the previous algorithm one should implement a condition testing if the dimension of  $\Phi_{N+1}$ differs from the one of $\Phi_N$. If it does not, the while loop should be terminated. 

One choice for $\Delta(Y_N,\mu)$ could be 
\begin{align*}
\Delta(Y_N,\mu):=\|u(\mu)-u_N(\mu)\|_U, 
\end{align*}
i.e. the error between the solution of $(\mathbb{P})$ and of $(\mathbb{P}_N)$. However, considering this choice in a linear search process over a very large training set $S_{\mbox{\scriptsize  train}}$  is computationally too costly, since the solution of the highly dimensional problem $(\mathbb{P})$ is needed.  In the next section we establish a choice for $\Delta(Y_N,\mu)$ that does not depend on the solution of $(\mathbb{P})$.


\section{A posteriori error analysis}
We start by associating to the solution  $(u_N,y_N,p_N)$ of (\ref{os_yN})--(\ref{os_uN}) at a given $\mu \in \mathcal{P}$ the function  $\tilde y \in Y$ that satisfies
\begin{align}
a(\tilde y,v;\mu)=b (u_N, v;\mu) + f(v;\mu)  \quad \forall \, v \in Y, \label{y_aux}
\end{align}
and the function $\tilde p \in Y$ such that 
\begin{align}
a(v,\tilde p;\mu)=(y_N-z,v)_{L^2(\Omega_0)} \quad \forall \, v \in Y. \label{p_aux}
\end{align}
Furthermore, we introduce the linear form $r_y(\cdot;\mu) \in Y^\ast$ defined by 
\begin{align*}
r_y(v;\mu):=b (u_N, v;\mu) + f(v;\mu) - a(y_N,v;\mu) \quad \forall \, v \in Y,
\end{align*}
and  $r_p(\cdot;\mu) \in Y^\ast$ by 
\begin{align*}
r_p(v;\mu):=(y_N-z,v)_{L^2(\Omega_0)}-a(v,p_N;\mu)  \quad \forall \, v \in Y. 
\end{align*}

We provide some estimates for $\tilde y$ and $\tilde p$ that will be utilized in the upcoming analysis. 

\begin{lemma}
Suppose that $(u,y,p)$ is the solution of (\ref{os_y})--(\ref{os_u}), and  $(u_N,y_N,p_N)$  the solution of (\ref{os_yN})--(\ref{os_uN}). Let $\tilde y$, $\tilde p$ be as defined in (\ref{y_aux}), (\ref{p_aux}), respectively. Then there holds
\begin{align}
 \|y-\tilde y\|_Y & \leq \frac{\kappa(\mu)}{\beta(\mu)}  \|u-u_N\|_U,  \label{y_yaux}\\
 \|p-\tilde p\|_Y & \leq \frac{1}{\rho_1^2 \beta(\mu)}\|y-y_N\|_{Y},  \label{p_paux} \\
 \frac{1}{\gamma(\mu)}  \|r_y(\cdot;\mu) \|_{Y^\ast}  & \leq  \|\tilde y-y_N\|_Y  \leq  \frac{1}{\beta(\mu)}  \|r_y(\cdot;\mu) \|_{Y^\ast} \label{yN_yaux}\\
 \frac{1}{\gamma(\mu)}  \|r_p(\cdot;\mu) \|_{Y^\ast}  & \leq  \|\tilde p-p_N\|_Y  \leq  \frac{1}{\beta(\mu)}  \|r_p(\cdot;\mu) \|_{Y^\ast} \label{pN_paux}
\end{align}
\end{lemma}
\begin{proof}
The proof is divided into four parts for clarity. In part (III) and (IV) of the proof we shall apply the estimating techniques  from \cite{haasdonk2017reduced} for linear elliptic PDEs. 

\textit{(I) Estimating $\|y-\tilde y\|_Y$:}
Using the coercivity of $a$, the continuity  of $b$,  together with (\ref{os_y}) and (\ref{y_aux}) gives
\begin{align*}
\beta(\mu) \|y-\tilde y\|_Y^2 &\leq a(y-\tilde y, y-\tilde y ;\mu) \\
&=b(u-u_N, y-\tilde y ;\mu) \\
& \leq \kappa(\mu) \|u-u_N\|_U \|y-\tilde y\|_Y, 
\end{align*}
from which (\ref{y_yaux}) follows after dividing both sides by $\beta(\mu) \|y-\tilde y\|_Y$.

\textit{(II) Estimating $\|p-\tilde p\|_Y$:} Similarly, but this time with (\ref{os_p}) and (\ref{p_aux}) we have 
\begin{align*}
\beta(\mu) \|p-\tilde p\|_Y^2 &\leq a(p-\tilde p, p-\tilde p ;\mu) \\
&=(y-y_N, p-\tilde p)_{L^2(\Omega_0)} \\
& \leq \|y-y_N\|_{H^1(\Omega)} \|p-\tilde p\|_{H^1(\Omega)} \\
& \leq \frac{1}{\rho_1^2}\|y-y_N\|_{Y} \|p-\tilde p\|_{Y}, 
\end{align*}
where we used (\ref{norm_eqv}). Dividing both sides by $\beta(\mu) \|p-\tilde p\|_Y$ gives (\ref{p_paux}).

\textit{(III) Estimating $\|\tilde y-y_N\|_Y$:} From the coercivity of $a$ and the definition of $r_y$, we have
\begin{align*}
\beta(\mu) \|\tilde y-y_N\|_Y^2 &\leq a(\tilde y-y_N, \tilde y-y_N ;\mu) \\
& = a(\tilde y, \tilde y-y_N ;\mu)- a(y_N, \tilde y-y_N ;\mu) \\
& = b (u_N, \tilde y-y_N;\mu) + f(\tilde y-y_N;\mu) - a(y_N, \tilde y-y_N ;\mu) \\
& = r_y(\tilde y-y_N;\mu) \\
& \leq \|r_y(\cdot;\mu) \|_{Y^\ast} \|\tilde y-y_N \|_{Y}, 
\end{align*}
which gives the upper bound in (\ref{yN_yaux}) after dividing both sides  by $\beta(\mu) \|\tilde y-y_N\|_Y$. 
On the other hand, let $v:= R r_y(\cdot;\mu)$ be the Riesz representative of $r_y(\cdot;\mu)$. Then using the continuity of $a$ it follows that
\begin{align*}
\|r_y(\cdot;\mu) \|_{Y^\ast}^2 &=\|v \|_{Y}^2=(v,v)_Y=r_y(v;\mu)=a(\tilde y-y_N, v ;\mu) \\
& \leq \gamma(\mu) \|\tilde y-y_N \|_{Y} \|v \|_{Y} \\
& =  \gamma(\mu) \|\tilde y-y_N \|_{Y} \|r_y(\cdot;\mu) \|_{Y^\ast}.
\end{align*}
Dividing both sides of the previous inequality by $\gamma(\mu) \|r_y(\cdot;\mu) \|_{Y^\ast}$ yields the lower bound in (\ref{yN_yaux}).

\textit{(IV) Estimating $\|\tilde p-p_N\|_Y$:} From the coercivity of $a$ and the definition of $r_p$ we have
\begin{align*}
\beta(\mu) \|\tilde p-p_N\|_Y^2 &\leq a(\tilde p-p_N, \tilde p-p_N ;\mu) \\
& =  a(\tilde p-p_N, \tilde p;\mu)- a(\tilde p-p_N, p_N ;\mu)  \\
& = (y_N-z,\tilde p-p_N)_{L^2(\Omega_0)} - a(\tilde p-p_N, p_N ;\mu)\\
& = r_p(\tilde p-p_N;\mu) \\
& \leq \|r_p(\cdot;\mu) \|_{Y^\ast} \|\tilde p-p_N \|_{Y}, 
\end{align*}
from which we deduce the upper bound in (\ref{pN_paux}) after dividing both sides by $\beta(\mu) \|\tilde p-p_N\|_Y$. On the other hand, let $v:= R r_p(\cdot;\mu)$ be  the Riesz representative of $r_p(\cdot;\mu)$. Then using the continuity  of $a$ we get 
\begin{align*}
\|r_p(\cdot;\mu) \|_{Y^\ast}^2 &=\|v \|_{Y}^2=(v,v)_Y=r_p(v;\mu)=a(v, \tilde p - p_N ;\mu) \\
& \leq \gamma(\mu) \|\tilde p - p_N \|_{Y} \|v \|_{Y} \\
& =  \gamma(\mu) \|\tilde p - p_N \|_{Y} \|r_p(\cdot;\mu) \|_{Y^\ast},
\end{align*}
Dividing both sides of the previous inequality by $\gamma(\mu)  \|r_p(\cdot;\mu) \|_{Y^\ast}$ gives the lower bound in (\ref{pN_paux}). This completes the proof.
\end{proof}

We now state our main result. It provides a posteriori estimator for the error in  approximating the solution of $(\mathbb{P})$ by the one of   $(\mathbb{P}_N)$. The estimator is sharp up to a constant. 

\begin{theorem}
\label{thm:uyp}
Suppose that $(u,y,p)$ is the solution of (\ref{os_y})--(\ref{os_u}), and  $(u_N,y_N,p_N)$ the solution of (\ref{os_yN})--(\ref{os_uN}).  Then there holds
\begin{align*}
\delta_{uyp}(\mu)\leq 
\|u-u_N\|_U + \|y-y_N\|_{Y} + \|p-p_N\|_{Y} 
\leq \Delta_{uyp}(\mu), 
\end{align*}
where 
\begin{align*}
\Delta_{uyp}(\mu):= & c_1(\mu) \|r_y(\cdot;\mu) \|_{Y^\ast} + c_2(\mu) \|r_p(\cdot;\mu) \|_{Y^\ast}, \\
\delta_{uyp}(\mu):= & c_3(\mu)  \|r_y(\cdot;\mu) \|_{Y^\ast} + c_4(\mu) \|r_p(\cdot;\mu) \|_{Y^\ast}, \\
c_1(\mu):=&\frac{1}{\beta(\mu)} \Big[\frac{1}{\rho_1 \sqrt{\alpha}} + \Big(1 +  \frac{1}{\rho_1^2 \beta(\mu)}\Big) \Big(\frac{ \kappa(\mu)}{\beta(\mu) \rho_1 \sqrt{\alpha}}+1 \Big) \Big],\\
c_2(\mu):=&\frac{1}{\beta(\mu)} \Big(\frac{\kappa(\mu)}{\alpha}  + \frac{\kappa(\mu)^2}{\beta(\mu) \alpha} + \frac{\kappa(\mu)^2}{\rho_1^2 \beta^2(\mu) \alpha}+1 \Big),\\
c_3(\mu):=&\frac{1}{2 \gamma(\mu)}\max \Big( \frac{\kappa(\mu) }{\beta(\mu)}, 1 \Big)^{-1}, \\
c_4(\mu):=&\frac{1}{2 \gamma(\mu)}\max \Big( \frac{1}{\rho_1^2 \beta(\mu)}, 1 \Big)^{-1}.
\end{align*}
\end{theorem}

\begin{proof}
The proof falls  into two parts, and we shall follow the ideas of \cite[Theorem~ 3.2]{gong2017adaptive} for adaptive finite element method for elliptic control problems. 

\textit{(I) Establishing an upper bound for $\|u-u_N\|_U + \|y-y_N\|_{Y} + \|p-p_N\|_{Y}$:}
Taking $v:=u_N$ in (\ref{os_u}), and $v:=u$ in (\ref{os_uN}), and adding the resulting inequalities, we get
\begin{align}
\alpha \|u-u_N\|_U^2 &\leq b(u_N-u,p-p_N;\mu)  \nonumber\\
& = b(u_N-u,p-\tilde p;\mu)+ b(u_N-u,\tilde p - p_N;\mu) \nonumber \\
& =: S_1 + S_2. \label{s1s2}
\end{align}
Recalling (\ref{norm_eqv}), an upper bound for $S_1$ can be obtained as follows. 
\begin{align*}
S_1 &=b(u_N-u,p-\tilde p;\mu) =a(\tilde y-y,p-\tilde p;\mu) =(y- y_N,\tilde y-y)_{L^2(\Omega_0)} \\
&=(y- y_N,\tilde y-y_N)_{L^2(\Omega_0)} -\|y- y_N\|^2_{L^2(\Omega_0)} \\
&\leq \frac{1}{2}\|\tilde y- y_N\|^2_{L^2(\Omega_0)} \leq  \frac{1}{2}\|\tilde y- y_N\|^2_{H^1(\Omega)}  \leq  \frac{1}{2 \rho^2_1}\|\tilde y- y_N\|^2_{Y}, 
\end{align*}
On the other hand, for $S_2$ we have
\begin{align*}
S_2 &=b(u_N-u,\tilde p-p_N;\mu) \\
& \leq   \frac{\alpha}{2}\|u_N-u\|_U^2 + \frac{1}{2 \alpha} \kappa(\mu)^2 \|\tilde p- p_N\|^2_{Y}.  
\end{align*}
Using the bounds of $S_1$ and $S_2$ in (\ref{s1s2}) yields
\begin{align}
\|u-u_N\|_U \leq \frac{1}{\rho_1 \sqrt{\alpha} } \|\tilde y- y_N\|_{Y} + \frac{1}{\alpha} \kappa(\mu) \|\tilde p- p_N\|_{Y}. \label{uuN}
\end{align}
Applying the triangle inequality, (\ref{y_yaux}), together with (\ref{uuN}) results in
\begin{align}
\|y-y_N\|_{Y} & \leq \|y-\tilde y\|_{Y} + \|\tilde y-y_N\|_{Y} \nonumber\\
& \leq \frac{\kappa(\mu)}{\beta(\mu)}  \|u-u_N\|_U + \|\tilde y-y_N\|_{Y} \nonumber\\
& \leq \Big(\frac{ \kappa(\mu)}{\beta(\mu) \rho_1 \sqrt{\alpha}}+1 \Big)\|\tilde y-y_N\|_{Y} +  \frac{\kappa(\mu)^2}{\beta(\mu) \alpha}  \|\tilde p- p_N\|_{Y}. \label{yyN}
\end{align}
Again the triangle inequality, (\ref{p_paux}), and (\ref{yyN}) yields to
\begin{align}
\|p-p_N\|_{Y} & \leq \|p-\tilde p\|_{Y} + \|\tilde p-p_N\|_{Y} \nonumber\\
& \leq \frac{1}{\rho_1^2 \beta(\mu) }  \|y-y_N\|_Y + \|\tilde p-p_N\|_{Y} \nonumber\\
& \leq \frac{1}{\rho_1^2 \beta(\mu) } \Big(\frac{ \kappa(\mu)}{\beta(\mu) \rho_1 \sqrt{\alpha}}+1 \Big)\|\tilde y-y_N\|_{Y} \nonumber \\
& \quad  +  \Big( \frac{\kappa(\mu)^2}{\rho_1^2 \beta^2(\mu) \alpha}+1 \Big)  \|\tilde p- p_N\|_{Y}. \label{ppN} 
\end{align}
Combining (\ref{uuN}), (\ref{yyN}), (\ref{ppN}), and recalling (\ref{yN_yaux}), (\ref{pN_paux}), we get
\begin{align*}
\|u-u_N\|_U + \|y-y_N\|_{Y} + \|p-p_N\|_{Y} \leq c_1(\mu) \|r_y(\cdot;\mu) \|_{Y^\ast} + c_2(\mu) \|r_p(\cdot;\mu) \|_{Y^\ast}, 
\end{align*}
where 
\begin{align*}
c_1(\mu):=&\frac{1}{\beta(\mu)} \Big[\frac{1}{\rho_1 \sqrt{\alpha}} + \Big(1 +  \frac{1}{\rho_1^2 \beta(\mu)}\Big) \Big(\frac{ \kappa(\mu)}{\beta(\mu) \rho_1 \sqrt{\alpha}}+1 \Big) \Big],\\
c_2(\mu):=&\frac{1}{\beta(\mu)} \Big(\frac{\kappa(\mu)}{\alpha}  + \frac{\kappa(\mu)^2}{\beta(\mu) \alpha} + \frac{\kappa(\mu)^2}{\rho_1^2 \beta^2(\mu) \alpha}+1 \Big).
\end{align*}

\textit{(II) Establishing a lower bound for $\|u-u_N\|_U + \|y-y_N\|_{Y} + \|p-p_N\|_{Y}$:}
From  (\ref{yN_yaux}), the triangle inequality, and  (\ref{y_yaux}) we have
\begin{align}
\frac{1}{\gamma(\mu)}  \|r_y(\cdot;\mu) \|_{Y^\ast} &\leq \|\tilde y- y_N\|_{Y} \leq  \|\tilde y- y\|_{Y}+ \| y- y_N\|_{Y} \nonumber\\
& \leq \frac{\kappa(\mu) }{\beta(\mu)} \|u-u_N\|_U + \|y-y_N\|_{Y} \nonumber \\
& \leq \max \Big( \frac{\kappa(\mu) }{\beta(\mu)}, 1 \Big) \Big(  \|u-u_N\|_U + \|y-y_N\|_{Y}  \Big). \label{r_y}
\end{align}
Similarly, but this time with (\ref{pN_paux}) and  (\ref{p_paux}) we get
\begin{align}
\frac{1}{\gamma(\mu)}  \|r_p(\cdot;\mu) \|_{Y^\ast} &\leq \|\tilde p- p_N\|_{Y} \leq \|\tilde p- p\|_{Y}+ \| p- p_N\|_{Y} \nonumber\\
&\leq \frac{1}{\rho_1^2 \beta(\mu)}\|y-y_N\|_{Y} + \| p- p_N\|_{Y} \nonumber\\
&\leq \max \Big( \frac{1}{\rho_1^2 \beta(\mu)}, 1 \Big) \Big( \|y-y_N\|_{Y} + \| p- p_N\|_{Y}\Big). \label{r_p}
\end{align}
From (\ref{r_y}) and (\ref{r_p}) one can  easily deduce that 
\begin{align*}
 c_3(\mu)  \|r_y(\cdot;\mu) \|_{Y^\ast} + c_4(\mu) \|r_p(\cdot;\mu) \|_{Y^\ast}  \leq 
\|u-u_N\|_U + \|y-y_N\|_{Y} + \|p-p_N\|_{Y}, 
\end{align*}
where 
\begin{align*}
c_3(\mu):=&\frac{1}{2 \gamma(\mu)}\max \Big( \frac{\kappa(\mu) }{\beta(\mu)}, 1 \Big)^{-1}, \\
c_4(\mu):=&\frac{1}{2 \gamma(\mu)}\max \Big( \frac{1}{\rho_1^2 \beta(\mu)}, 1 \Big)^{-1}.
\end{align*}
This concludes the proof. 
\end{proof}

Next, we establish a posteriori estimator for the relative error of the controls. 

\begin{corollary}
\label{thm:u}
Under the hypothesis of Theorem~\ref{thm:uyp}, there holds
\begin{align}
\frac{\|u-u_N\|_U}{\|u\|_U } \leq \frac{2 \Delta_u(\mu)}{\|u_N\|_U } \label{rel_err}
\end{align}
provided that $\frac{2\Delta_u(\mu)}{\|u_N\|_U} \leq 1$, where 
\begin{align*}
\Delta_u(\mu):=\frac{1}{\rho_1 \sqrt{\alpha} \beta(\mu)}  \|r_y(\cdot;\mu) \|_{Y^\ast}  + \frac{ \kappa(\mu) }{\alpha  \beta(\mu)}  \|r_p(\cdot;\mu) \|_{Y^\ast}. 
\end{align*}
\end{corollary}
\begin{proof}
From the estimate (\ref{uuN}) combined with  (\ref{yN_yaux}) and (\ref{pN_paux}) we obtain 
\begin{align}
\|u-u_N\|_U \leq \Delta_u(\mu). \label{ineq1}
\end{align}
Let $\frac{2\Delta_u(\mu)}{\|u_N\|_U} \leq 1$, then we have
\begin{align*}
\big|   \|u\|_U - \|u_N\|_U\big| \leq \|u-u_N\|_U \leq \Delta_u(\mu) \leq \frac{1}{2} \|u_N\|_U.
\end{align*}
It follows from the previous inequality that if $\|u_N\|_U \geq \|u\|_U$, then 
\begin{align}
\frac{1}{2} \|u_N\|_U \leq  \|u\|_U  \label{ineq2}
\end{align}
which is clearly valid  also when $\|u_N\|_U < \|u\|_U$. Thus, from (\ref{ineq1}) and (\ref{ineq2}) the desired estimate   (\ref{rel_err}) can be deduced.  
\end{proof}

\begin{remark}
\label{error_gap}
The term $S_1$ in the proof of Theorem~\ref{thm:uyp} is over estimated by dropping the term $-\frac{1}{2}\|y- y_N\|^2_{L^2(\Omega_0)}$, consequently, so is the term (\ref{uuN}). By this, a gap of a noticeable size should be expected between the relative error of controls and its a posteriori estimator  in (\ref{rel_err}). 
\end{remark}

\begin{remark}
To consider the upper bound $\Delta_u(\mu)$  from Corollary~\ref{thm:u}  or $\Delta_{uyp}(\mu)$ from Theorem~\ref{thm:uyp} for $\Delta(Y_N,\mu)$ in Algorithm~\ref{algo:greedy},  the constants $\kappa(\mu)$ and $\beta(\mu)$ should be generally replaced by other ones, say  $\tilde{\kappa}(\mu)$ and $\tilde{\beta}(\mu)$, respectively, that are  computationally cheaper to evaluate. In particular, we assume that 
\begin{align*}
\beta(\mu) \geq \tilde{\beta}(\mu) \geq \beta_0  \quad \forall \, \mu \in \mathcal{P},\\
\kappa(\mu) \leq \tilde{\kappa}(\mu) \leq \kappa_0  \quad \forall \, \mu \in \mathcal{P}. 
\end{align*}
Such constants $\tilde{\kappa}(\mu)$ and $\tilde{\beta}(\mu)$ can be obtained using, for instance, the min-theta approach after assuming parameter-separability for the bilinear forms $a$ and $b$,  see \cite{haasdonk2017reduced} for the details. 
\end{remark}

\section{Convergence analysis}
In this section we are concerned with the question of whether the solution of the reduced control problem $(\mathbb{P}_N)$ converges to the solution of $(\mathbb{P})$ as $N \to \infty$. For this purpose, we need to investigate the continuity with respect to the parameter $\mu$ and the uniform boundedness with respect to $N$ for the quantities that appear during the analysis.

For a given $u \in U$, we introduce the mapping 
\begin{align}
S_u: \mathcal{P}\to Y  \label{Smap}
\end{align}
such that the function $y \in Y$, $y:=S_u(\mu)$,  is the solution of the variational problem (\ref{wpde}) corresponding to $u \in U$ and $\mu \in \mathcal{P}$. By Lax-Milgram's lemma, the mapping (\ref{Smap}) is well defined. 

In what follows we set
\begin{multline*}
\|a(\cdot,\cdot,\mu)-a(\cdot,\cdot,\xi)\|_A:= \sup_{\|v\|_Y=1,\|w\|_Y=1}|a(v,w,\mu)-a(v,w,\xi)|, \\ \quad \|f(\cdot,\mu)-f(\cdot,\xi)\|_F:= \sup_{\|v\|_Y=1}|f(v,\mu)-f(v,\xi)|, \text{ and}\\
\|b(\cdot,\cdot,\mu)-b(\cdot,\cdot,\xi)\|_B:= \sup_{\|v\|_U=1,\|w\|_Y=1}|b(v,w,\mu)-b(v,w,\xi)|.
\end{multline*}

\begin{lemma}
\label{lemma:S}
For a given $u \in U$, let $S_u$ be the mapping defined in (\ref{Smap}). Then the following estimates hold;

\begin{align}
\|S_u(\mu)\|_Y  \leq  c_0(\|u\|_U  +1)   \qquad \forall \,  \mu \in \mathcal{P}, \label{uniformbound} 
\end{align}
where $c_0:=\frac{1}{\beta_0}\max(\kappa_0,\sigma_0)$, and 
\begin{align}
&\|S_u(\mu_2)-S_u(\mu_1)\|_Y   \leq \frac{1}{\beta_0}\Big( c_0 \|a(\cdot,\cdot;\mu_2)-a(\cdot,\cdot;\mu_1)\|_A (\|u\|_U  + 1) \nonumber \\
& \quad + \|b(\cdot,\cdot;\mu_2)-b(\cdot,\cdot;\mu_1)\|_B \|u\|_U + \|f(\cdot;\mu_2)-f(\cdot;\mu_1)\|_F \Big), \label{Scontinuous}
\end{align}
for any $\mu_1, \mu_2 \in \mathcal{P}$.

\end{lemma}
\begin{proof}
To prove (\ref{uniformbound}), we denote $y:=S_u(\mu)$.  From the coerciveness of the bilinear form $a(\cdot,\cdot; \mu)$ together with the boundedness of $b(\cdot,\cdot;\mu)$ and $f(\cdot;\mu)$, one obtains 
\begin{align*}
\beta_0 \|y\|^2_Y & \leq a(y,y;\mu) =b(u,y;\mu)+f(y;\mu) \\
  & \leq \kappa_0 \|u\|_U \|y\|_Y + \sigma_0 \|y\|_Y, \\
  & \leq \max(\kappa_0,\sigma_0) (\|u\|_U  + 1)\|y\|_Y, 
\end{align*}
which gives (\ref{uniformbound}) after dividing in  the previous inequality both sides by  $\beta_0 \|y\|_Y$.

To verify (\ref{Scontinuous}) we define $y_1:=S_u(\mu_1)$ and $y_2:=S_u(\mu_2)$. Employing the coerciveness of $a(\cdot,\cdot; \mu_1)$ and the estimate (\ref{uniformbound}), we get
\begin{align*}
\beta_0 \|y_1-y_2\|^2_Y &\leq a(y_1-y_2, y_1-y_2;\mu_1) \\
&=b(u,y_1-y_2;\mu_1)+f(y_1-y_2;\mu_1) -a(y_2, y_1-y_2;\mu_1)  \\
&=b(u,y_1-y_2;\mu_1)+f(y_1-y_2;\mu_1) -a(y_2, y_1-y_2;\mu_1) \\
& \quad + a(y_2, y_1-y_2;\mu_2)-b(u,y_1-y_2;\mu_2)-f(y_1-y_2;\mu_2) \\
& \leq \|a(\cdot,\cdot;\mu_2)-a(\cdot,\cdot;\mu_1)\|_A \|y_2\|_Y \|y_1-y_2\|_Y \\
& \quad + \|b(\cdot,\cdot;\mu_2)-b(\cdot,\cdot;\mu_1)\|_B \|u\|_U \|y_1-y_2\|_Y \\
& \quad + \|f(\cdot;\mu_2)-f(\cdot;\mu_1)\|_F  \|y_1-y_2\|_Y \\
& \leq c_0 \|a(\cdot,\cdot;\mu_2)-a(\cdot,\cdot;\mu_1)\|_A (\|u\|_U  + 1) \|y_1-y_2\|_Y \\
& \quad + \|b(\cdot,\cdot;\mu_2)-b(\cdot,\cdot;\mu_1)\|_B \|u\|_U \|y_1-y_2\|_Y \\
& \quad + \|f(\cdot;\mu_2)-f(\cdot;\mu_1)\|_F  \|y_1-y_2\|_Y, 
\end{align*}
from which one deduces (\ref{Scontinuous}) after dividing both sides of the inequality by $\beta_0 \|y_1-y_2\|_Y$.  
\end{proof}

We associate to the reduced variational problem (\ref{Nwpde}) the mapping 
\begin{align}
S_{N,u}:\mathcal{P} \to Y_N, \label{SNmap}
\end{align} 
where the function $y_N \in Y_N$, $y_N:=S_{N,u}(\mu)$, is the solution of (\ref{Nwpde}) corresponding to the given  $u \in U$ and $\mu \in \mathcal{P}$.

\begin{lemma}
For a given $u \in U$, let $S_{N,u}$ be the mapping defined in (\ref{SNmap}). Then the following estimates hold

\begin{align*}
\|S_{N,u}(\mu)\|_Y  \leq  c_0(\|u\|_U  +1)   \qquad \forall \,  \mu \in \mathcal{P}, 
\end{align*}
where $c_0:=\frac{1}{\beta_0}\max(\kappa_0,\sigma_0)$, and 
\begin{align*}
&\|S_{N,u}(\mu_2)-S_{N,u}(\mu_1)\|_Y   \leq \frac{1}{\beta_0}\Big( c_0 \|a(\cdot,\cdot;\mu_2)-a(\cdot,\cdot;\mu_1)\|_A (\|u\|_U  + 1) \nonumber \\
& \quad + \|b(\cdot,\cdot;\mu_2)-b(\cdot,\cdot;\mu_1)\|_B \|u\|_U + \|f(\cdot;\mu_2)-f(\cdot;\mu_1)\|_F \Big), 
\end{align*}
for any $\mu_1, \mu_2 \in \mathcal{P}$.

\end{lemma}
\begin{proof}
A long the lines of Lemma~\ref{lemma:S}'s proof.
\end{proof}

\begin{theorem}
\label{thm:ubound}
Let $\bar u(\mu) \in U_{ad}$ be the solution of $(\mathbb{P})$ for an arbitrary  $\mu \in \mathcal{P}$. Then, there exists a constant $c>0$ independent of $\mu$ such that there holds 
\begin{align}
\|\bar u(\mu)\|_U\leq c \big( \|z\|_{L^2(\Omega_0)} + \|u\|_U +1\big)  \quad \forall \, u \in U_{ad}.\label{uniformbound2}
\end{align}
\end{theorem}
\begin{proof}
For a given $\mu \in \mathcal{P}$, let $u \in U_{ad}$ be an arbitrary feasible control with the corresponding state $y(\mu)$, and let $\bar y(\mu) \in Y$ denote the state associated with the optimal control $\bar u(\mu)$. Then, the optimality of $\bar u$ implies 
\begin{align*}
\frac{\alpha}{2}\|\bar u\|^2_U & \leq J(\bar u)=\frac{1}{2}\|\bar y-z\|^2_{L^2(\Omega_0)}+\frac{\alpha}{2}\|\bar u\|^2_U \\
& \leq J(u)=\frac{1}{2}\|y-z\|^2_{L^2(\Omega_0)}+\frac{\alpha}{2}\| u\|^2_U \\
& \leq \|y\|^2_{L^2(\Omega_0)}+\|z\|^2_{L^2(\Omega_0)}  +\frac{\alpha}{2}\| u\|^2_U \\
& \leq \|y\|^2_{H^1(\Omega)}+\|z\|^2_{L^2(\Omega_0)}  +\frac{\alpha}{2}\| u\|^2_U \\
& \leq \frac{1}{\rho^2_1}\|y\|^2_{Y}+\|z\|^2_{L^2(\Omega_0)}  +\frac{\alpha}{2}\| u\|^2_U \\
& \leq \frac{c^2_0}{\rho^2_1}\big(\|u\|_{U}+1\big)^2+\|z\|^2_{L^2(\Omega_0)}  +\frac{\alpha}{2}\| u\|^2_U, 
\end{align*}
where (\ref{norm_eqv}) and (\ref{uniformbound}) are used in the last two inequalities, respectively. Taking the square root of  both sides of the previous inequality gives the desired result. 
\end{proof}

\begin{theorem}
Let $\bar u_N(\mu) \in U_{ad}$ be the solution of $(\mathbb{P}_N)$ for an arbitrary  $\mu \in \mathcal{P}$. Then, there exists a constant $c>0$ independent of $\mu$ or $N$ such that there holds 
\begin{align*}
\|\bar u_N(\mu)\|_U\leq c \big( \|z\|_{L^2(\Omega_0)} + \|u\|_U +1\big)  \quad \forall \, u \in U_{ad}.
\end{align*}
\end{theorem}
\begin{proof}
A long the lines of Theorem~\ref{thm:ubound}'s proof. 
\end{proof}

\begin{theorem}
\label{thm:u1u2}
Let $u(\mu)\in U_{ad}$ be the solution of $(\mathbb{P})$ corresponding to some given $\mu \in \mathcal{P}$. Then, for any $\mu_1, \mu_2 \in \mathcal{P}$ the following estimate holds
\begin{align*}
\|u(\mu_1)-u(\mu_2)\|_U \leq c \sqrt{ \|a(\mu_2)-a(\mu_1)\|_A + \|b(\mu_2)-b(\mu_1)\|_B  + \|f(\mu_2)-f(\mu_1)\|_F }
\end{align*}
for some $c>0$ independent of $\mu_1$ and $\mu_2$. Here $a(\mu):=a(\cdot,\cdot;\mu)$, $b(\mu):=b(\cdot,\cdot;\mu)$ and $f(\mu):=f(\cdot;\mu)$ for any $\mu \in \mathcal{P}$.

\end{theorem}
\begin{proof}
Let $u_1:=u(\mu_1)$ and $u_2:=u(\mu_2)$. According to Theorem~\ref{thm:oc}, the optimal triple $(u_1,y_1,p_1) \in U_{ad} \times Y \times Y$ satisfies 
\begin{align}
& a(y_1,v;\mu_1)=b (u_1, v;\mu_1) + f(v;\mu_1)  \quad \forall \, v \in Y, \label{os_y1}\\
& a(v,p_1;\mu_1)=(y_1-z,v)_{L^2(\Omega_0)} \quad \forall \, v \in Y, \label{os_p1}\\
& b(v-u_1,p_1;\mu_1)  + \alpha (u_1, v-u_1)_U \geq 0 \quad \forall \, v \in U_{ad}, \label{os_u1}
\end{align}
while $(u_2,y_2,p_2) \in U_{ad} \times Y \times Y$ satisfies 
\begin{align}
& a(y_2,v;\mu_2)=b (u_2, v;\mu_2) + f(v;\mu_2)  \quad \forall \, v \in Y, \label{os_y2}\\
& a(v,p_2;\mu_2)=(y_2-z,v)_{L^2(\Omega_0)} \quad \forall \, v \in Y, \label{os_p2}\\
& b(v-u_2,p_2;\mu_2)  + \alpha (u_2, v-u_2)_U \geq 0 \quad \forall \, v \in U_{ad}. \label{os_u2}
\end{align}
We shall utilize the auxiliary functions $\tilde y_1, \tilde y_2 \in Y$ satisfying 
\begin{align*}
a(\tilde y_1, v; \mu_2 ) & =b(u_1, v ; \mu_2 ) +f(v;\mu_2) \quad \forall \, v \in Y, \\
a(\tilde y_2, v; \mu_1) & =b(u_2, v ; \mu_1) +f(v;\mu_1) \quad \forall \, v \in Y.
\end{align*}

Testing (\ref{os_u1}) against $u_2$, and (\ref{os_u2}) against $u_1$, and adding the resulting inequalities yields 
\begin{align*}
\alpha \|u_1-u_2\|^2_U & \leq b(u_2-u_1,p_1 ; \mu_1)+ b(u_1-u_2, p_2 ; \mu_2)\\
& =  b(u_2,p_1 ; \mu_1)-  b(u_1,p_1 ; \mu_1)+ b(u_1, p_2 ; \mu_2)-  b(u_2, p_2 ; \mu_2) \\
& = b(u_2,p_1 ; \mu_1) -a(y_1,p_1;\mu_1)+ f(p_1;\mu_1) \\
& \quad +  b(u_1,p_2 ; \mu_2) -a(y_2,p_2;\mu_2)+ f(p_2;\mu_2) \\
& =a(\tilde y_2-y_1,p_1;\mu_1) + a(\tilde y_1-y_2,p_2;\mu_2) \\
&=(y_1-z,\tilde y_2-y_1)_{L^2(\Omega_0)} + (y_2-z,\tilde y_1-y_2)_{L^2(\Omega_0)} \\
&=(y_1-z,\tilde y_2-y_2)_{L^2(\Omega_0)} + (y_2-z,\tilde y_1-y_1)_{L^2(\Omega_0)} -\|y_1-y_2\|^2_{L^2(\Omega_0)}\\
& \leq \|y_1-z\|_{L^2(\Omega_0)} \|\tilde y_2-y_2\|_{L^2(\Omega_0)} + \|y_2-z\|_{L^2(\Omega_0)} \|\tilde y_1-y_1\|_{L^2(\Omega_0)}\\
& \leq c \Big(\|u_1\|_{U}+ \|z\|_{L^2(\Omega_0)} +1\Big) \|\tilde y_2-y_2\|_{L^2(\Omega_0)} \\
& \quad + c \Big(\|u_2\|_{U}+ \|z\|_{L^2(\Omega_0)} +1\Big) \|\tilde y_1-y_1\|_{L^2(\Omega_0)},  
\end{align*}
where  (\ref{uniformbound}) is used in the last inequality. We proceed by utilizing (\ref{Scontinuous}) 
\begin{align*}
&\leq c \Big(\|u_1\|_{U}+ \|z\|_{L^2(\Omega_0)} +1\Big)  \Big(  \|a(\mu_2)-a(\mu_1)\|_A (\|u_2\|_U  + 1) \\
& \quad + \|b(\mu_2)-b(\mu_1)\|_B \|u_2\|_U + \|f(\mu_2)-f(\mu_1)\|_F \Big) \\
& \quad +  c \Big(\|u_2\|_{U}+ \|z\|_{L^2(\Omega_0)} +1\Big)  \Big(  \|a(\mu_2)-a(\mu_1)\|_A (\|u_1\|_U  + 1) \\
& \quad + \|b(\mu_2)-b(\mu_1)\|_B \|u_1\|_U + \|f(\mu_2)-f(\mu_1)\|_F \Big) \\
& \leq c \Big(  \|a(\mu_2)-a(\mu_1)\|_A  + \|b(\mu_2)-b(\mu_1)\|_B + \|f(\mu_2)-f(\mu_1)\|_F \Big)
\end{align*}
Recalling (\ref{uniformbound2}) and taking the square root of the both sides gives the desired result. 
\end{proof}

\begin{theorem}
\label{thm:uN1uN2}
Let $u_N(\mu) \in U_{ad}$ be the solution of $(\mathbb{P}_N)$ corresponding to some given $\mu\in \mathcal{P}$. Then, for any $\mu_1, \mu_2 \in \mathcal{P}$ the following estimate holds
\begin{align*}
\|u_N(\mu_1)-u_N(\mu_2)\|_U \leq c \sqrt{ \|a(\mu_2)-a(\mu_1)\|_A + \|b(\mu_2)-b(\mu_1)\|_B  + \|f(\mu_2)-f(\mu_1)\|_F }
\end{align*}
for some $c>0$ independent of $\mu_1, \mu_2$ or $N$. Here $a(\mu):=a(\cdot,\cdot;\mu)$, $b(\mu):=b(\cdot,\cdot;\mu)$ and $f(\mu):=f(\cdot;\mu)$ for any $\mu \in \mathcal{P}$.

\end{theorem}
\begin{proof}
A long the lines of Theorem~\ref{thm:u1u2}'s proof. 
\end{proof}

Recall that the space $Y_N$ considered in $(\mathbb{P}_N)$ is constructed from the snapshots $\{y(\mu^1),p(\mu^1), \ldots, y(\mu^N),p(\mu^N)\}$ taken from $(\mathbb{P})$ at the sample parameters $\{\mu^1, \ldots, \mu^N \}=:\mathcal{P}_N \subset \mathcal{P}$. We denote 
\begin{align*}
h_N:=\max_{\mu \in \mathcal{P}}\min_{\mu' \in \mathcal{P}_N} \|\mu - \mu'\|
\end{align*}
with $\|\cdot\|$ being  the Euclidean norm in $\mathbb{R}^p$. We shall assume that $0<h_N \leq 1$ and that  as $N \to \infty$, $h_N \to 0$, i.e. the set $\mathcal{P}_N$ gets denser in $\mathcal{P}$ as $N$ increases. Furthermore, it is natural to assume that for any $\mu \in \mathcal{P}_N$ there holds 
\begin{align*}
u_N(\mu)=u(\mu),
\end{align*}
where $u_N(\mu)$ and $u(\mu)$ denote the solutions of $(\mathbb{P}_N)$ and  $(\mathbb{P})$, respectively, at the given $\mu$ since the mapping $\mathcal{P}\ni \mu \mapsto u_N(\mu) \in U$ is supposed to interpolate the mapping  $\mathcal{P}\ni \mu \mapsto u(\mu) \in U$ at the set of parameters $\mathcal{P}_N$. Finally, we assume that for any $\mu_1, \mu_2 \in \mathcal{P}$ we have
\begin{align*}
\|a(\mu_2)-a(\mu_1)\|_A &\leq c \|\mu_2-\mu_1\|^{q_a},\\
\|b(\mu_2)-b(\mu_1)\|_B &\leq c \|\mu_2-\mu_1\|^{q_b},\\
\|f(\mu_2)-f(\mu_1)\|_F &\leq c \|\mu_2-\mu_1\|^{q_f},
\end{align*}
for some $c, q_a, q_b, q_f >0$ independent of $\mu_1$ or $\mu_2$ where $\|\cdot\|$ denotes the Euclidean norm in $\mathbb{R}^p$,  i.e. the bilinear forms $a$ and $b$ and the linear form $f$ are continuous in $\mu$. Under these assumptions, we formulate the next theorem. 

\begin{theorem}\label{convergence}
Let $u_N(\mu), u(\mu) \in U$ denote the solutions of $(\mathbb{P}_N)$ and  $(\mathbb{P})$, respectively, for a given $\mu \in \mathcal{P}$. Then, the following estimate holds 
\begin{align*}
\|u_N(\mu)-u(\mu)\|_U \leq c h^t_N
\end{align*}
where $t:=\frac{1}{2}\min\{q_a, q_b, q_f\}$ and $c>0$ is a constant independent of $h_N$ or $\mu$.
\end{theorem}
\begin{proof}
Let $\mu \in \mathcal{P}$ be given, and let $\mu^\ast:=\mbox{arg min}_{\mu' \in \mathcal{P}_N}\|\mu-\mu'\|$. Then, recalling Theorem~\ref{thm:u1u2}, Theorem~\ref{thm:uN1uN2}, the fact that $u_N(\mu^\ast)=u(\mu^\ast)$, and the continuity of $a, b$ and $f$ gives
\begin{align*}
\|u(\mu)-u_N(\mu)\|_U & \leq \|u(\mu)-u(\mu^\ast)\|_U +\|u(\mu^\ast)-u_N(\mu^\ast)\|_U+\|u_N(\mu^\ast)-u_N(\mu)\|_U\\
& \leq c \sqrt{ \|a(\mu)-a(\mu^\ast)\|_A + \|b(\mu)-b(\mu^\ast)\|_B  + \|f(\mu)-f(\mu^\ast)\|_F }\\
& \leq c \sqrt{ \|\mu-\mu^\ast\|^{q_a} + \|\mu-\mu^\ast\|^{q_b}  + \|\mu-\mu^\ast\|^{q_f} }\\
& \leq c \sqrt{ h^{q_a}_N + h^{q_b}_N  + h^{q_f}_N } \leq c h^t_N
\end{align*} 
where $t:=\frac{1}{2}\min\{q_a, q_b, q_f\}$.
\end{proof}

\section{Numerical examples}
In this section we apply our theoretical findings to construct numerically reduced surrogates for  two  examples, namely a thermal block problem and a Graetz flow problem, which are taken from \cite{bader2016certified}. In particular, we discretize those two examples using variational discretization, then we build their reduced counterparts using the greedy procedure from Algorithm~\ref{algo:greedy}, where we use the bound $2\Delta_u(\mu)/\|u_N\|_U$ from Corollary~\ref{thm:u} for the estimator $\Delta(Y_N,\mu)$. Finally, we compare the solutions of the reduced problems to their corresponding ones from the highly dimensional problems to asses the quality of the obtained reduced models.

\begin{example}[Thermal block]
\label{thermalblock}
We consider the control problem 
\[
\min_{(u,y) \in U_{ad} \times Y} J(u,y)=\frac{1}{2} \| y-z \|_{L^2(\Omega_0)}^2  + \frac{\alpha}{2} \| u \|_{U}^2  
\]
subject to 
\[
\mu \int_{\Omega_1} \nabla y \cdot \nabla v \, dx + \int_{\Omega_2} \nabla y \cdot \nabla v \, dx= \int_{\Omega}  u v \, dx   \quad \forall \, v \in Y, 
\]
\end{example}
\noindent where
\begin{align*}
&\Omega_1:=(0,0.5)\times (0,1), \quad \Omega_2:=(0.5,1)\times (0,1), \quad \Omega:=\Omega_1 \cup \Omega_2,\\
& \Omega_0:=\Omega, \quad z(x)=1 \mbox{ in }\Omega, \quad  U:=L^2(\Omega), \quad (\cdot,\cdot)_U:=(\cdot,\cdot)_{L^2(\Omega)}, \\
& U_{ad}:=\{u \in L^2(\Omega): u(x)\geq u_a(x) \mbox{ a.e } x \in \Omega\}, \quad u_a(x):=2+2(x_1-0.5), \\
& \mu \in \mathcal{P}:=[0.5, 3], \quad \alpha = 10^{-2}, 
\end{align*}
and the space $Y \subset H^1_0(\Omega) \cap C(\bar{\Omega})$ is the space of piecewise linear and continuous finite elements endowed with the inner product $(\cdot,\cdot)_Y:=(\nabla \cdot, \nabla \cdot)_{L^2(\Omega)}$. The underlying PDE admits a homogeneous Dirichlet boundary condition on the boundary  $\partial\Omega$ of the domain $\Omega$.   

From the previous given data, it is an easy task to see that (\ref{norm_eqv}) holds with $\rho_2=1$ and $\rho_1=\frac{1}{\sqrt{c_p^2+1}}$ where $c_p=\frac{1}{\sqrt{2}\pi }$ is the  Poincar\'{e}'s constant  in the inequality 
\[
 \| v \|_{L^2(\Omega)} \leq c_p \| \nabla v \|_{L^2(\Omega)} \quad \forall \, v \in H^1_0(\Omega).
\]
Furthermore, we  take $\tilde \kappa(\mu)=c_p$ and $\tilde \beta(\mu)=\min(\mu,1)$.

We use a uniform triangulation for $\Omega$ such that $\mbox{dim}(Y) \approx 8300$. The solution of both the variational discrete control problem and the reduced control problem for a given parameter $\mu$  is achieved by solving the corresponding optimality conditions using a  semismooth Newton's method with the stopping criteria 
\[
\frac{1}{\alpha} \| p^{(k)} -p^{(k+1)} \|_{L^2(\Omega)} \leq 10^{-11},
\]
where $p^{(k)}$ is the adjoint variable at the $k$-th iteration.

The reduced space  $Y_N$ for the considered problem was constructed employing the greedy procedure introduced in  Algorithm~\ref{algo:greedy} with the choice $S_{\mbox{\scriptsize  train}}:=\{s_j\}_{j=1}^{100}$, $s_j:=0.5 \times (3/0.5)^{(j-1)/99}$, $\mu^1:=0.5$,  $\varepsilon_{\mbox{\scriptsize tol}} =10^{-8}$, $N_{\mbox{\scriptsize max}}=30$,  and $\Delta(Y_N,\mu):=2 \Delta_u(\mu)/\|u_N\|_{L^2(\Omega)}$.

The algorithm terminated  before reaching the prescribed tolerance $\varepsilon_{\mbox{\scriptsize tol}} $   and that was after 22 iterations as it could not enrich the reduced basis with any new linearly independent samples. To investigate the quality of the obtained reduced basis and the sharpness of the employed upper bound $\Delta(Y_N,\mu)$, we compute the maximum of the relative error   $\|u-u_N\|_{L^2(\Omega)}/\|u\|_{L^2(\Omega)}$ and of the corresponding bound  $2 \Delta_u(\mu)/\|u_N\|_{L^2(\Omega)}$ over the set $S_{\mbox{\scriptsize  test}}:=\{s_j\}_{j=1}^{125}$, $s_j:=0.503\times(2.99/0.503)^{(j-1)/125}$ for the greedy algorithm iterations $N=1,\ldots,22$. The graphical illustration is presented in Figure~\ref{figure:example_1}. We see that the error decays  dramatically in the first nine iterations, namely it drops from 1 to slightly above $10^{-6}$, then the decay  becomes very slow and the error almost stabilizes at $10^{-6}$ in the last four iterations. As predicted in Remark~\ref{error_gap}, we can see a gap between the relative error and the used estimator $\Delta(Y_N,\mu)$. 
This plot compares to Fig.1(b) of \cite{bader2016certified}. We observe that four iterations of the greedy algorithm with our approach deliver the same error reduction as thirty iterations of the greedy algorithm in \cite{bader2016certified}. A similar behaviour is observed for Example \ref{graetzflow} with the Graetz flow in Figure \ref{figure:example_2}, which compares to the results documented in Fig. 3(b) of \cite{bader2016certified}. For this example six iterations of the greedy algorithm with our approach deliver the same error reduction as thirty iterations of the greedy algorithm in \cite{bader2016certified}.

\begin{figure}[p]
	\centering
	\includegraphics[trim = 40mm 85mm 40mm 90mm, width=1\textwidth]{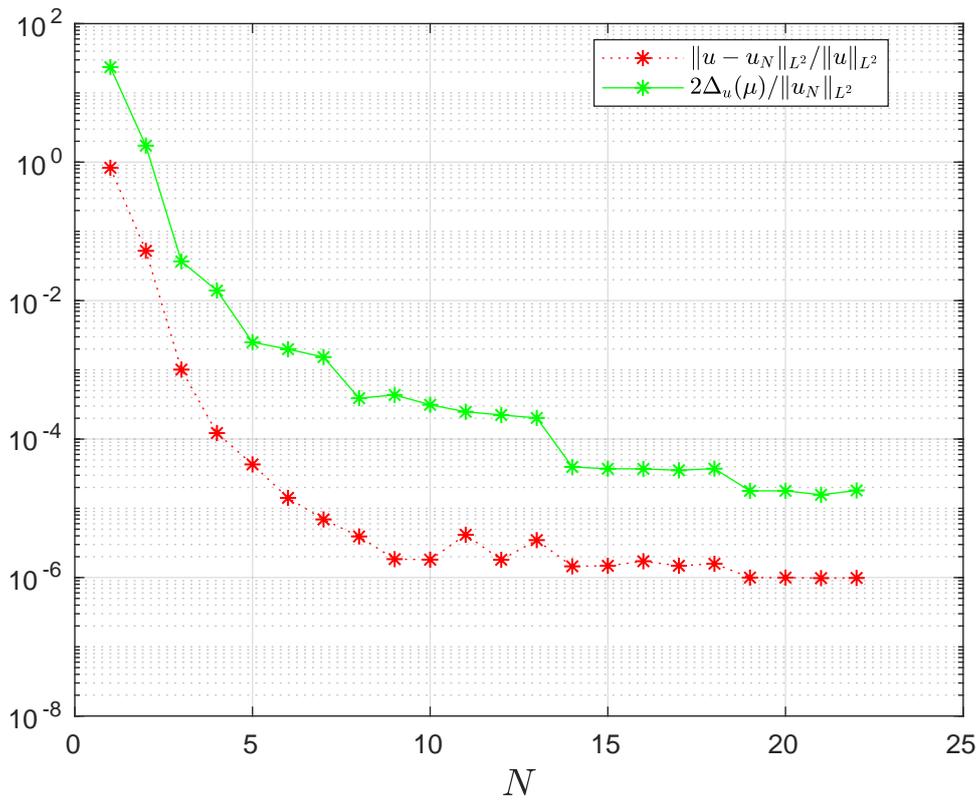}
	\caption{Example~\ref{thermalblock}: The maximum of $\|u-u_N\|_{L^2(\Omega)}/\|u\|_{L^2(\Omega)}$ the relative error of controls   and the corresponding  upper bounds $2 \Delta_u(\mu)/\|u_N\|_{L^2(\Omega)}$ over $S_{\mbox{\scriptsize  test}}$ versus the greedy algorithm iterations $N=1,\ldots,22$.}
	\label{figure:example_1}
\end{figure}

\begin{example}[Graetz flow] We consider the problem
\label{graetzflow}
\[
\min_{(u,y) \in U_{ad}(\mu_2) \times Y(\mu_2)} J(u,y)=\frac{1}{2} \| y-z \|_{L^2(\Omega_0(\mu_2))}^2  + \frac{\alpha}{2} \| u \|_{U(\mu_2)}^2  
\]
subject to 
\[
\frac{1}{\mu_1} \int_{\Omega(\mu_2)} \nabla y \cdot \nabla v \, dx + \int_{\Omega(\mu_2)} \beta(x)\cdot\nabla y   v \, dx= \int_{\Omega(\mu_2)}  u v \, dx   \quad \forall \, v \in Y(\mu_2), 
\]
\end{example}
\noindent with the data
\begin{align*}
&\Omega(\mu_2):=(0,1.5+\mu_2)\times (0,1), \quad \Omega_1(\mu_2):=(0.2\mu_2,0.8\mu_2)\times (0.3,0.7), \\ &\Omega_2(\mu_2):=(\mu_2+0.2,\mu_2+1.5)\times (0.3,0.7),  \quad \Omega_0(\mu_2):=\Omega_1(\mu_2)\cup\Omega_2(\mu_2)\\
& \beta(x)=(x_2(1-x_2),0)^T \mbox{ in }\Omega(\mu_2), \quad  z(x)=0.5 \mbox{ in }\Omega_1(\mu_2),  \quad z(x)=2 \mbox{ in }\Omega_2(\mu_2) \\
& U(\mu_2):=L^2(\Omega(\mu_2)), \quad (\cdot,\cdot)_{U(\mu_2)}:=(\cdot,\cdot)_{L^2(\Omega(\mu_2))}, \\
& U_{ad}(\mu_2):=\{u \in L^2(\Omega(\mu_2)): u(x)\geq u_a(x) \mbox{ a.e } x \in \Omega(\mu_2)\},  u_a(x):=-0.5, \\
& (\mu_1,\mu_2) \in \mathcal{P}:=[5, 18]\times[0.8,1.2], \quad \alpha = 10^{-2},
\end{align*}
and  $ Y(\mu_2) \subset \{v \in H^1(\Omega(\mu_2)) \cap C(\overline{\Omega(\mu_2)}):v|_{\Gamma_D(\mu_2)}=1\}$ is the space of piecewise linear and continuous finite elements. The underlying PDE has the homogeneous Neumann boundary condition $\partial_{\eta}y|_{\Gamma_N(\mu_2)}=0$ on the portion $\Gamma_N(\mu_2)$ of the boundary of the domain $\Omega(\mu_2)$ , and the Dirichlet boundary condition $y|_{\Gamma_D(\mu_2)}=1$ on the portion $\Gamma_D(\mu_2)$. An illustration for the domain $\Omega(\mu_2)$ and the boundary segments  $\Gamma_D(\mu_2)$ and $\Gamma_N(\mu_2)$ is given in  Figure~\ref{figure:domain}.

We introduce the lifting function $\tilde y(x):=1$ to handle the nonhomogeneous Dirichlet boundary condition, and reformulate the problem  over the reference domain $\Omega:=\Omega(\mu_2^{\mbox{\scriptsize ref}})$, and endow the state space $Y:=Y(\mu_2^{\mbox{\scriptsize ref}})$  by the inner product $(\cdot,\cdot)_Y$ given by
\[
(v,w)_Y:=\frac{1}{\mu_1^{\mbox{\scriptsize ref}}} \int_{\Omega} \nabla w \cdot \nabla v \, dx + \frac{1}{2}\Big(\int_{\Omega} \beta(x)\cdot\nabla w   v \, dx+ \int_{\Omega} \beta(x)\cdot\nabla v   w \, dx \Big)
\]
where $(\mu_1^{\mbox{\scriptsize ref}},\mu_2^{\mbox{\scriptsize ref}})=(5,1)$. The control space $U:=U(\mu_2^{\mbox{\scriptsize ref}})$ is endowed with a parameter dependent inner product  $(\cdot,\cdot)_{U(\mu_2)}$ from the affine geometry transformation, see \cite{rozza2007reduced}. After transforming the problem over $\Omega$ we deduce that (\ref{norm_eqv}) holds with $\rho_1=\max(\mu_1^{\mbox{\scriptsize ref}}(1+c_p),1)^{-2}$, where the constant $c_p$ is from the Poincar\'{e}'s inequality 
\[
\int_\Omega v^2 \, dx \leq c_p \int_\Omega |\nabla v|^2\, dx \quad \forall \, v \in H^1(\Omega): v|_{\Gamma_D(\mu_2^{\mbox{\scriptsize ref}})}=0.
\]
In addition, we take
\[
\tilde{\beta}(\mu_1, \mu_2)=\min\Big(\mu_1^{\mbox{\scriptsize ref}}\min(\frac{1}{\mu_1 \mu_2},\frac{\mu_2}{\mu_1},\frac{1}{\mu_1}),1 \Big), \mbox{ and } \tilde{\kappa}(\mu_1, \mu_2)=\frac{1}{\rho_1}(\sqrt{\mu_2}+1).
\]

The domain $\Omega$ is partitioned via a uniform triangulation such that $\mbox{dim}(Y) \approx 4900$. The optimality conditions corresponding to the variational discrete control problem and the reduced control problem are solved using a semismooth Newton's method with the stopping criteria 
\[
\frac{1}{\alpha} \| p^{(k)} -p^{(k+1)} \|_{U(\mu_2)} \leq 10^{-11},
\]
where $p^{(k)}$ is the adjoint variable at the $k$-th iteration.

The optimal controls and their active sets for the parameter values $(\mu_1,\mu_2)=(5,0.8)$, $(18,1.2)$ computed on the reference domain are presented in Figure~\ref{figure:controls}.

The reduced basis for the space  $Y_N$ is  constructed applying the Algorithm~\ref{algo:greedy} with the choice  $S_{\mbox{\scriptsize  train}}:=\{(s^1_j,s^2_k)\}$ for $j,k=1,\ldots,30$ where $s^1_j:=5 \times (18/5)^{(j-1)/29}$ and $s^2_k:=(0.4/29)\times(k-1)+0.8$. Furthermore, we take $\mu^1:=(5,0.8)$,  $\varepsilon_{\mbox{\scriptsize tol}} =10^{-8}$, $N_{\mbox{\scriptsize max}}=30$,  and $\Delta(Y_N,\mu):=2 \Delta_u(\mu)/\|u_N\|_{U(\mu_2)}$.

The algorithm terminated at $N_{\mbox{\scriptsize max}}=30$  before reaching the  tolerance $\varepsilon_{\mbox{\scriptsize tol}} $. To asses the quality of the resulting reduced basis and the sharpness of the bound $\Delta(Y_N,\mu)$, 
we compare the maximum of the relative error   $\|u-u_N\|_{U(\mu_2)}/\|u\|_{U(\mu_2)}$ to the  bound  $2 \Delta_u(\mu)/\|u_N\|_{U(\mu_2)}$ computed over the test set $S_{\mbox{\scriptsize  test}}:=\{(s^1_j,s^2_k)\}$, for $j=1,\ldots, 10$ and $k=1, \ldots, 5$ where $s_j:=5.2\times(17.5/5.2)^{(j-1)/9}$, and $s^2_k:=(0.35/4)\times(k-1)+0.82$ for the greedy algorithm iterations $N=1,\ldots,30$. The outcome of the experiment  is presented in Figure~\ref{figure:example_2}. The error decay is of moderate speed in comparison to the previous example. It could be because the current problem has more parameters and one of which stems from the geometry of the domain. We again see the gap between the bound and the error, which supports the prediction of  Remark~\ref{error_gap}.

\section{Conclusions}
With present a reduced basis method for the approximation of optimal control problems with control constraints. We use variational discretization from \cite{hinze2005variational} for the numerical approximation of the optimal control problems. This allows us to use methods from \cite{gong2017adaptive} to prove an error equivalence for our residual based error estimator, which finally is one of the key ingredients for the convergence proof of our approach in Theorem \ref{convergence}. Our numerical results indicate that the reduced basis method combined with variational discretization for a prescribed error tolerance seems to deliver reduced basis spaces of much smaller dimension than in the existing approaches reported in the literature, compare e.g. the numerical results reported in  \cite{bader2016certified}. However, this comes along with a more sophisticated numerical implementation of the variational discretization approach in the case of control constraints, for which the classical offline-online decomposition techniques are not applicable in a straightforward manner.

\begin{figure}[p]
	\centering
	\includegraphics[trim = 80mm 145mm 50mm 90mm, width=0.6\textwidth]{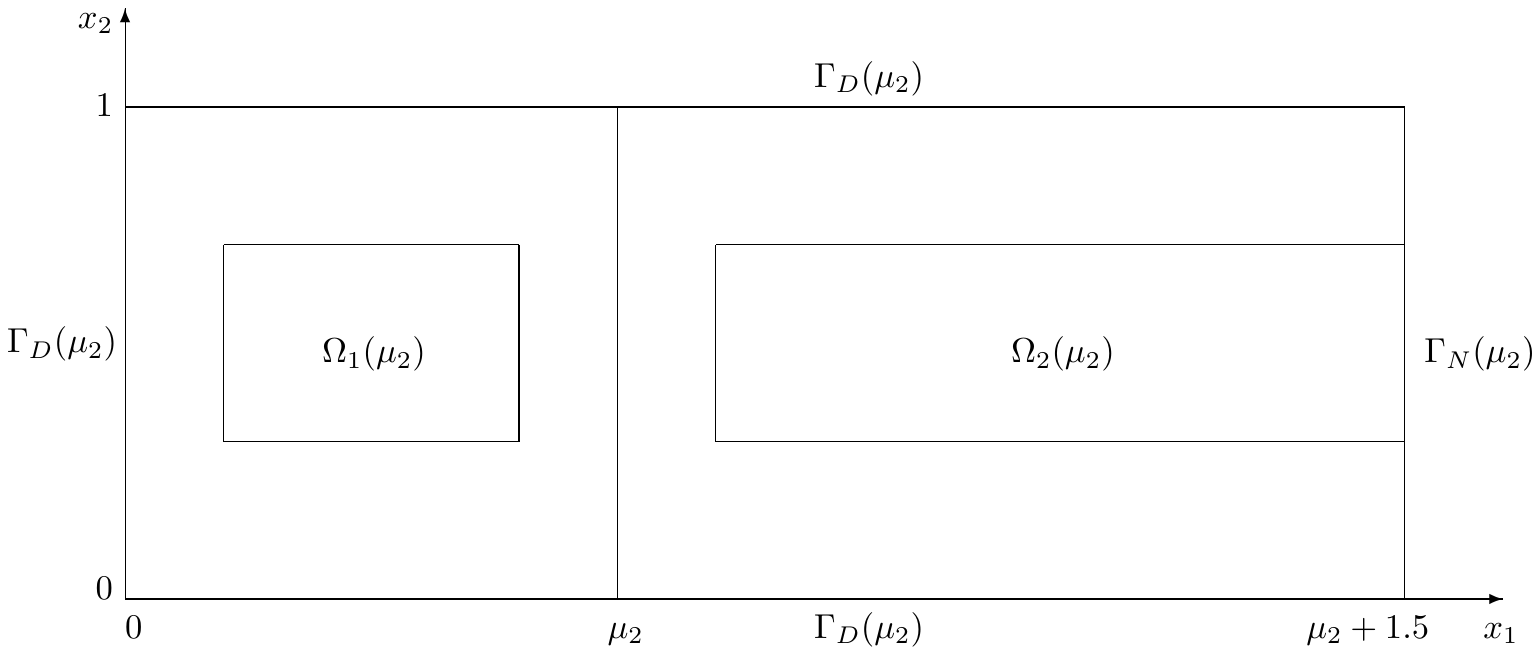}
	\caption{Example~\ref{graetzflow}: The domain $\Omega(\mu_2)$ for the Graetz flow problem.}
	\label{figure:domain}
\end{figure}

\begin{figure}
    \centering
    \begin{subfigure}[b]{0.75\textwidth}
        \includegraphics[trim = 80mm 100mm 50mm 100mm, width=0.5\textwidth]{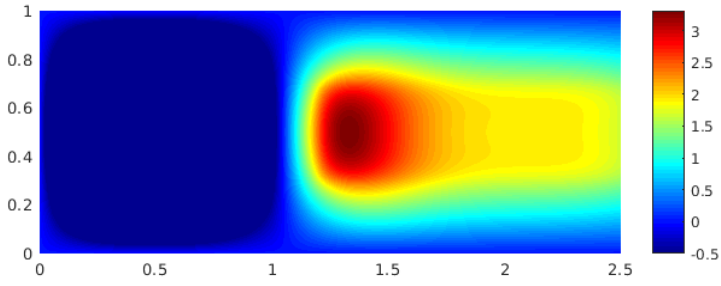}
        \caption{The optimal control for \\$(\mu_1,\mu_2)=(5,0.8)$}
    \end{subfigure}~ 
    \begin{subfigure}[b]{0.75\textwidth}
        \includegraphics[trim = 80mm 100mm 50mm 100mm, width=0.5\textwidth]{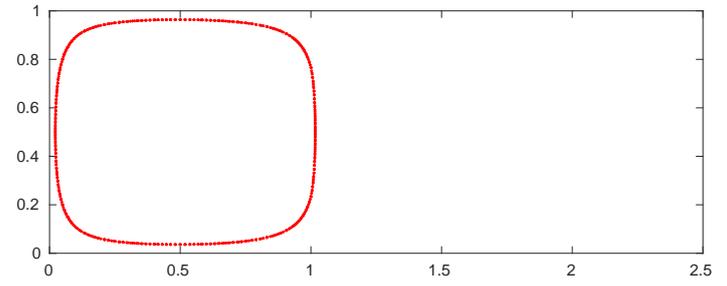}
     \caption{The control active set boundary \\for $(\mu_1,\mu_2)=(5,0.8)$}
    \end{subfigure}
     \begin{subfigure}[b]{0.75\textwidth}
            \includegraphics[trim = 80mm 100mm 50mm 100mm, width=0.5\textwidth]{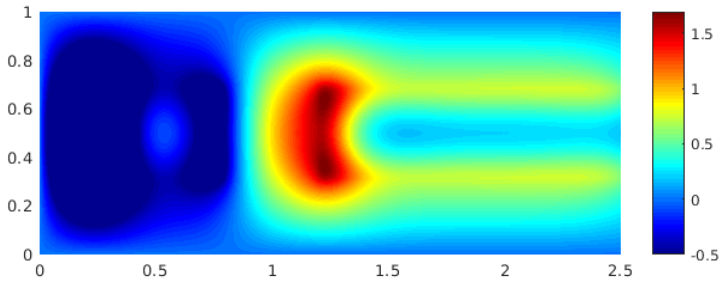}
             \caption{The optimal control for \\$(\mu_1,\mu_2)=(18,1.2)$}
        \end{subfigure}~
     \begin{subfigure}[b]{0.75\textwidth}
            \includegraphics[trim = 80mm 100mm 50mm 100mm, width=0.5\textwidth]{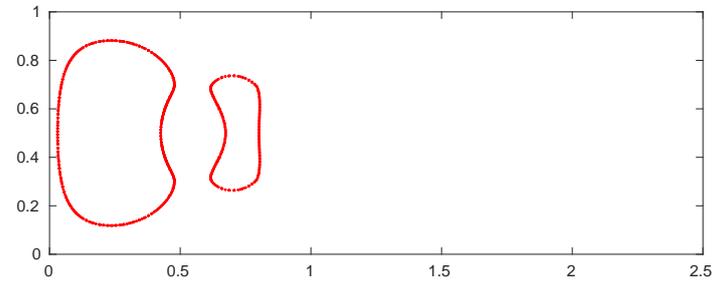}
            \caption{The control active set boundary \\for $(\mu_1,\mu_2)=(18,1.2)$}
        \end{subfigure}
    \caption{Example~\ref{graetzflow}: The optimal controls, and their active sets (enclosed by the curves) for $(\mu_1,\mu_2)=(5,0.8), \mbox{ and } (18,1.2)$ computed on the reference domain $\Omega$.}\label{figure:controls}
\end{figure}

\begin{figure}[p]
	\centering
	\includegraphics[trim = 40mm 85mm 40mm 90mm, width=1\textwidth]{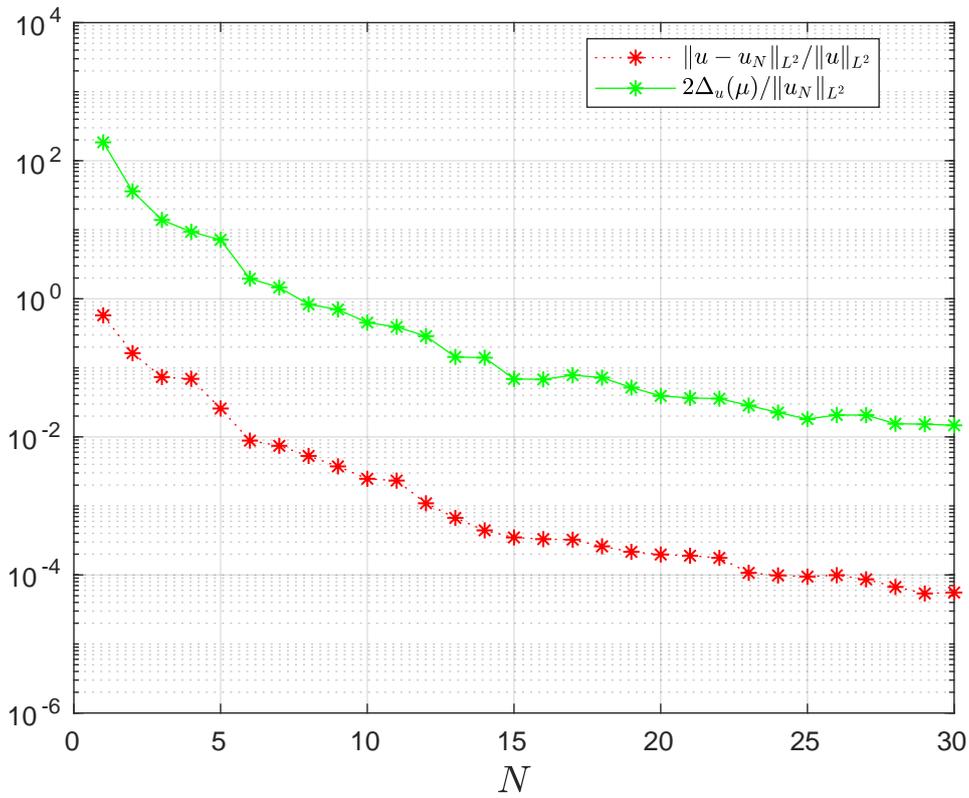}
	\caption{Example~\ref{graetzflow}: The maximum of $\|u-u_N\|_{U(\mu_2)}/\|u\|_{U(\mu_2)}$ the relative error of controls   and the  upper bound $2 \Delta_u(\mu)/\|u_N\|_{U(\mu_2)}$ over $S_{\mbox{\scriptsize  test}}$ versus the greedy algorithm iterations $N=1,\ldots,30$.}
	\label{figure:example_2}
\end{figure}

\newpage
\bibliographystyle{plain}
\bibliography{references}
\end{document}